\documentclass[11pt]{amsart}

\usepackage[utf8]{inputenc}

\usepackage{fullpage}

\usepackage{enumitem}
\usepackage{amsfonts}
\usepackage{amsmath}
\usepackage{amsthm}
\usepackage{amssymb}
\usepackage{mathrsfs}
\usepackage{enumitem}
\usepackage{forest}
\usepackage{hyperref}
 
\usepackage[all]{xy}
\usepackage{appendix}

\usepackage{tikz}

\usepackage{
amscd,
color,
epsfig,
float,
graphicx,
latexsym,
mathrsfs,
mathtools,
psfrag,
transparent,
url,
wasysym,
wrapfig,
tikz,
tikz-cd,
xr
}

\usetikzlibrary{shapes.geometric}
\usetikzlibrary{positioning}
\usetikzlibrary{automata}
\usetikzlibrary{decorations}
\usetikzlibrary{decorations.pathreplacing}

\usepackage{tikz-cd}

\newtheorem{prop}{Proposition}[section]

\newtheorem{lem}[prop]{Lemma}
\newtheorem{theo}[prop]{Theorem}

\theoremstyle{definition}
\newtheorem{defi}{Definition}[section]
\newtheorem{eg}[defi]{Example}
\newtheorem{rem}[defi]{Remark}

\newcommand{\Coll}{\mathsf{Coll}}

\newcommand{\PQ}{\overline{P \cup Q}}

\newcommand{\bR}{\mathbb{R}}

\newcommand\bx{\mathbf{x}}
\newcommand\by{\mathbf{y}}

\newcommand{\on}{\operatorname}

\newcommand{\Irr}{{\on{Irr}}}

\title{Constrainahedra}
\author{Nathaniel Bottman}
\address{Max Planck Institute for Mathematics,
Vivatsgasse 7, 53111 Bonn, Germany}
\email{\href{mailto:bottman@mpim-bonn.mpg.de}{bottman@mpim-bonn.mpg.de}}
\author{Daria Poliakova}
\address{Max Planck Institute for Mathematics,
Vivatsgasse 7, 53111 Bonn, Germany
\break\indent
Center for Quantum Mathematics, University of Southern Denmark, Campusvej 55, Odense M, DK-5230, Denmark}
\email{\href{mailto:polydarya@gmail.com}{polydarya@gmail.com}}

\begin{document}
\maketitle

\begin{abstract}
We define a family of convex polytopes called \emph{constrainahedra}, which index collisions of horizontal and vertical lines.
Our construction proceeds by first defining a poset $C(m,n)$ of \emph{good rectangular preorders}, then proving that $C(m,n)$ is a lattice, and finally constructing a polytopal realization by taking the convex hull of a certain explicitly-defined collection of points.
The constrainahedra will form the combinatorial backbone of the second author's construction of \emph{strong homotopy duoids}.
We indicate how constrainahedra could be realized as Gromov-compactified configuration spaces of horizontal and vertical lines; viewed from this perspective, the constrainahedra include naturally into the first author's notion of 2-associahedra.
\end{abstract}

\section{Introduction}

In this paper, we define a family of polytopes called \emph{constrainahedra}, which generalize associahedra and multiplihedra.
The intended application of the constrainahedra is the definition of strong homotopy duoids.
While this definition is still in progress in work of the second author, the constrainahedra are natural objects that are of independent interest (c.f.\ e.g.\ \cite{chapoton_pilaud}).

Consider a configuration of $n$ horizontal and $m$ vertical lines, depicted here with $n=3, m=5$:

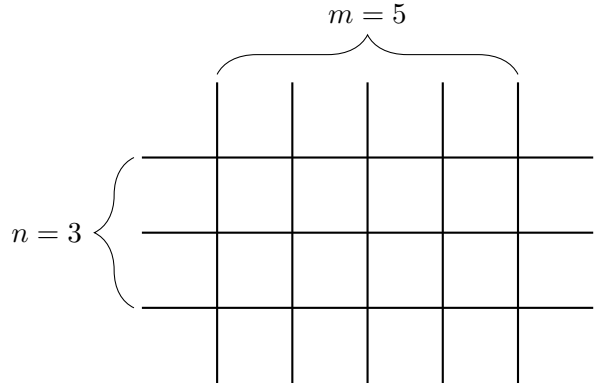
\begin{center}
\begin{tikzpicture}

\draw[thick] (0,0) -- (0,-4);
\draw[thick] (1,0) -- (1,-4);
\draw[thick] (2,0) -- (2,-4);
\draw[thick] (3,0) -- (3,-4);
\draw[thick] (4,0) -- (4,-4);

\draw[thick] (-1,-1) -- (5,-1);
\draw[thick] (-1,-2) -- (5,-2);
\draw[thick] (-1,-3) -- (5,-3);

\draw[decorate,decoration={brace,amplitude=15pt}, yshift = 3pt] (0,0) -- (4,0) node[midway,yshift = 23pt] {$m = 5$};

\draw[decorate,decoration={brace,amplitude=15pt}, xshift = -3pt] (-1,-3) -- (-1,-1) node[midway,xshift = -33pt] {$n = 3$};

\end{tikzpicture} 
\end{center}

\noindent
In \S\ref{sec:main_def}, we define posets $C(m,n)$ which index collisions of horizontal and vertical lines in the arrangement we have just depicted.
The constrainahedra are polytopal realizations of these posets; we abuse notation and refer to both $C(m,n)$ and its polytopal realization as the \emph{$(m,n)$-th constrainahedron}.
In \S\ref{sec:lattice}, we prove that $C(m,n)$ is a lattice, and in \S\ref{sec:polytopes}, we construct a polytopal realization of $C(m,n)$ by taking the convex hull of certain points.
Finally, in \S\ref{sec:comparison}, we compare this polytopal realization to a different realization that comes from work of Chapoton--Pilaud (c.f.\ also \S\ref{ss:history}).

\subsection{Towards strong homotopy duoids}
Recall from \cite{batanin_markl} that a {\em duoidal category} is a category $C$ equipped with two associative tensor products 
\begin{align}
\otimes
\colon
C \times C \to C,
\qquad
\boxtimes\colon
C \times C \to C
\end{align}
with (probably different) units for $\otimes$ and $\boxtimes$, and a natural (not necessarily invertible) morphism 
\begin{align}
(A \boxtimes B) \otimes  (C \boxtimes D) \to (A \otimes C) \boxtimes (B \otimes D)
\end{align}

An example of a duoidal category is the category of bimodules over a Hopf algebra $A$, where the first tensor product is the bimodule tensor product over $A$, and the second tensor product is induced by comultiplication. 

In a duoidal category, a {\em duoid} is an object $V$ equipped with two associative multiplications $V \otimes V \to V$ and $V \boxtimes V \to V$, such that the following commutes

\[
\begin{tikzcd}
 (V \boxtimes V) \otimes (V \boxtimes V) \arrow{r} \arrow{d}  & V \otimes V \arrow{dd} \\
 (V \otimes V) \boxtimes (V \otimes V) \arrow{d} & \\
 V \boxtimes V \arrow{r} & V
\end{tikzcd}
\]

\noindent
A strict monoidal category is an example of a duoid (in a duoidal category of above type). Now suppose that our duoidal category is a DG-category.
Similarly to how the notion of monoid can be weakened to the notion of strong homotopy monoid by means of associahedra, one may ask what is a strong homotopy duoid.
This question is still yet to be answered, but the current paper represents the first step: indeed, we intend constrainahedra to do for duoids what associahedra do for monoids.
Once we have defined strong homotopy duoids, they will give a very explicit model for weak monoidality.

\subsection{Constrainahedra as Gromov-compactified configuration spaces}
\label{ss:gromov}

In a series of papers \cite{bottman:figure_eight, 
bottman_wehrheim,
bottman:2-associahedra,
bottman:realization,
bottman_carmeli}, the first author and his collaborators have developed a theory of functoriality for the Fukaya category.
This theory is based on a family of abstract polytopes called \emph{2-associahedra}, which index degenerations in configuration spaces of vertical pointed lines in $\bR^2$, modulo translations and dilations.
In \cite{bottman:realization}, the first author showed that these posets can be realized as stratified topological spaces, by considering Gromov compactifications of the configuration spaces mentioned in the previous sentence.
We expect that the constrainahedra can be realized in a completely analogous way, and in future work we plan to construct such a realization.
The interior of this realization is easy to describe: it is the configuration space of grids in $\bR^2$, modulo translations in both coordinates and dilations that scale both coordinates equally.

\begin{rem}
If we considered instead the direct product of two associahedra, the interior would be identical except that the dilations we quotient by would not be required to scale both coordinates equally.
(C.f. \S\ref{sec:comparison}.)
\null\hfill$\triangle$
\end{rem}

\subsection{The history of the constrainahedra}
\label{ss:history}

We now briefly comment on the history of the construction of the constrainahedra.
The first description of the constrainahedra appeared in a 2016 private communication from the second author to Patrick Tierney.
The constrainahedra appeared in Tierney's bachelor thesis \cite{tierney}.
As of 2016, the constrainahedra had only been defined in very rough terms.
This remained the case for several years, because the first author did not see a use for the constrainahedra in symplectic geometry.

During her PhD, the second author realized that the constrainahedra should form the combinatorial backbone of the notion of a strong homotopy duoid.
This motivated her to define the constrainahedra in a precise fashion.
Following her construction, Chapoton--Pilaud produced an alternate construction in \cite{chapoton_pilaud}, as a special case of a more general notion of \emph{shuffle products of generalized permutahedra}.
In this language, the constrainahedra are the result of shuffling the associahedra with themselves.

\subsection{Acknowledgements}

We are grateful to Sergey Arkhipov, Spencer Backman, Richard Stanley, and James Stasheff for useful conversations.
In particular, the vertex coordinates that we define in \S\ref{sec:polytopes} came out of ongoing joint work with Spencer Backman.

N.B.\ was supported by an NSF Standard Grant (DMS-1906220) during the preparation of this article.
N.B.\ and D.P.\ are grateful to the Max Planck Institute for Mathematics in Bonn for its hospitality and financial support.
D.P.\ acknowledges the support of the Danish National Research Foundation (Grant DNRF157).

\section{Main definition}
\label{sec:main_def}

In this section, we will define the poset $C(m,n)$.

\subsection{Good rectangular preorders}

Definition~\ref{def:constrainahedra}, in which we define constrainahedra, will rely crucially on the notion of \emph{good rectangular preorders}.
Toward this notion, we fix some notation.
Denote by $L_1, \ldots, L_n$ the horizontal lines and by $M_1, \ldots, M_m$ the vertical lines.
Let $\Coll(m,n)$ be the finite set consisting of elements $m_i$ for $1 \leq i < m$ and $l_i$ for $1 \leq j < n$.
The element $m_i$ will be formally called {\em the vertical collision between lines $M_i$ and $M_{i+1}$}, and the element $l_j$ will be formally called {\em the horizontal collision between lines $L_j$ and $L_{j+1}$}.

Our main objects of study
will be certain {\em preorders} on the set $\Coll(m,n)$. Recall the following definition.

\begin{defi}
A {\em preorder} on a set $X$ is a binary relation that is reflexive and transitive, but not necessarily anti-symmetric.
\null\hfill$\triangle$
\end{defi}

\noindent
In an ordered set, two distinct elements can satisfy one of the following: $x<y$, $x>y$, or $x\#y$ (incomparable).
In a preordered set, there is also a fourth possibility: $x \equiv y$ (corresponding to $x \leq y$ and $y \leq x$ both holding, which would be impossible for two distinct elements in an ordered set).
We now say that $x$ and $y$ are comparable if $x<y$ or $x>y$ or $x \equiv y$.

We will deal with certain preorders on the set $\Coll(m,n)$, called {\em rectangular} preorders.
Let us fix the following terminology.
\begin{itemize}
    \item
    Collisions $m_i$ and $l_j$ are {\em orthogonal}.
    
    \item
    Collisions $m_i$ and $m_{i'}$ (or $l_j$ and $l_{j'}$) are called {\em parallel}.
    
    \item
    For a fixed preorder, a collision $l_s$ is called an {\em orthogonal link} between two parallel collisions $m_i$ and $m_j$ if $m_i \leq l_s \leq m_j$.
    (The definition of an orthogonal link between $l_i$ and $l_j$ is similar.)
    
    \item
    A collision $m_s$ is called a {\em gap} between two parallel collisions $m_i$ and $m_j$ if $s \in [i,j]$ and $m_i < m_s > m_j$.
    (The definition of a gap between $l_i$ and $l_j$ is similar.)
\end{itemize}

We are now ready to give the main definition.

\begin{defi}
A rectangular preorder on $\Coll(m,n)$ is {\em good} if it satisfies:

\begin{enumerate}
    \item
    {\sc (Orthogonal comparability) }
    Orthogonal collisions are always comparable.
    
    \item
    {\sc (Parallel comparability) }
    Parallel collisions are comparable if and only if at least one of the following holds:
    
    \begin{itemize}
        \item
        There is an orthogonal link between them.
        
        \item
        There is no gap between them.
        \null\hfill$\triangle$
    \end{itemize}
\end{enumerate}
\end{defi}

Having a good rectangular preorder, we read $x < y$ as ``collision $x$ happened earlier then collision $y$, we read $x \equiv y$ as ``collision $x$ happened simultaneously with collision $y$,'' and we read $x \# y$ as ``collision $x$ happened far from collision $y$'' (so we have no idea which of them was first).

\begin{eg}
\label{preord}
This is a good rectangular preorder on the set $\Coll(3,5)$, depicted in terms of its Hasse diagram:

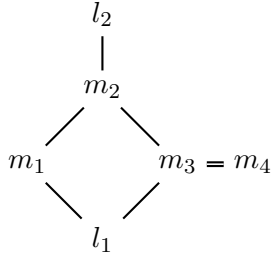
\begin{center}
\begin{tikzpicture}
\node (1) {$l_2$};
\node (2) [below of = 1] {$m_2$};
\node (the) [below of = 2] {};
\node (3) [left of = the] {$m_1$};
\node (4) [right of = the] {$m_3$};
\node (5) [below of = the] {$l_1$};
\node (6) [right of = 4] {$m_4$};

\draw[thick] (1) -- (2);
\draw[thick] (2) -- (3);
\draw[thick] (2) -- (4);
\draw[thick] (4) -- (5);
\draw[thick] (3) -- (5);
\draw[thick, double] (4) -- (6);

\end{tikzpicture}
\end{center}

\noindent
According to this preorder, in particular, $m_2$ happened earlier than $l_2$, $m_1$ and $m_3$ happened far from each other, and $m_3$ and $m_4$ happened simultaneously. 
\null\hfill$\triangle$
\end{eg}

For two rectangular preorders $P_1$ and $P_2$, we say $P_1 \leq P_2$ if $P_1$ refines $P_2$.
(This condition is equivalent to the statement that if $x \leq_{P_1} y$, then $x \leq_{P_2} y$.)
If we view preorders as subsets of $\Coll(m,n) \times \Coll(m,n)$ consisting of pairs $(x,y)$ such that $x \leq y$, then refinement corresponds to the opposite of inclusion.

We now introduce the main definition of this paper.

\begin{defi}
\label{def:constrainahedra}
The constrainahedron $C(m,n)$ is the poset of good rectangular preorders on $\Coll (m,n)$.
\null\hfill$\triangle$
\end{defi}

The constructed family of posets generalizes two known families.

\begin{theo}
Posets $C(1,n) = C(n,1)$ coincide with face posets of associahedra, and $C(2,n) = C(n,2)$ coincide with face posets of multiplihedra.
\end{theo}

\begin{proof}
To pass from $C(n,1)$ to associahedra, use the labelling of faces by planar trees. Consider a planar tree with $n$ leaves.
To obtain a preorder on $\Coll(n,1)$, we associate to each $m_i$ an inner vertex $v(m_i)$ located between leaves $i$ and $i+1$.
Then we say $m_i \leq m_{i+1}$ if there is a descending path from $v(m_i)$ to $v(m_{i+1})$.
{\sc (Parallel comparability)} corresponds to inner vertices being comparable if and only if they belong to the same branch.

For the other direction, use the labelling of faces by bracketings of lines $M_1$, $\ldots$, $M_n$. Having a good rectangular preorder, for each collision $m_i$ add a bracket embracing the lines that have collided through collisions that happened earlier than $m_i$.
{\sc (Parallel comparability)} ensures that these brackets are well-defined. The procedures described above are inverse to each other.

To compare $C(n,2)$ with multiplihedra, consider some preorder on $\Coll(n,2)$.
Restricting to $\Coll(n,2) \setminus \{l_1 \}$ and forgetting the comparisons that existed due to an orthogonal link, we obtain a planar tree as explained above.
But now for every inner vertex (corresponding to a collision $m_i$) we have an extra piece of data: whether it happened before $l_1$, simultaneously with $l_1$ or after $l_1$ (by {\sc (orthogonal comparability)}, all $m_i$ are comparable to $l_1$).
This gives the painting (recall that due to Forcey \cite{forcey} faces of multiplihedra are represented by painted trees). 
\end{proof}

\subsection{Anna \& Bob metaphor for preorders}
Informally, a good rectangular preorder corresponds to an account of a collision happening in discrete linear time, but with some data lost due to limitations of observation.
Let us put an observer on each of the lines except for boundary lines, and additionally let us put an observer in each of the squares.
For $2 \times 4$ case, we need two Annas to sit on $M_2$ and $M_3$, and three Bobs to sit in the squares:

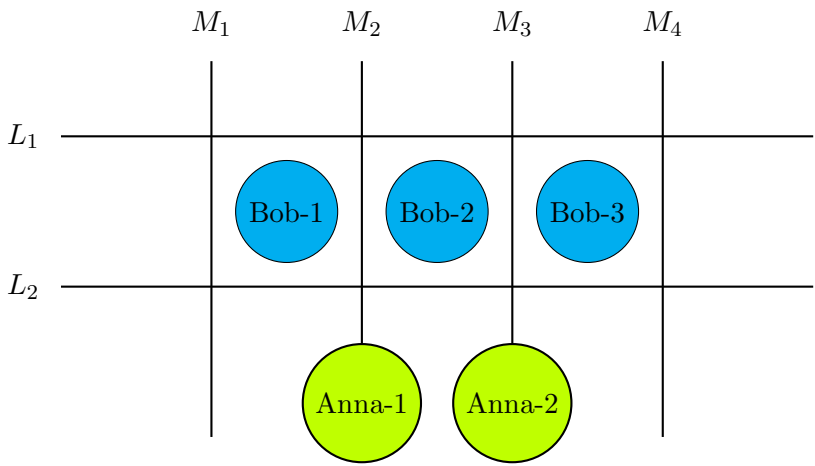
\begin{center}
\begin{tikzpicture}
\draw[thick] (0,-1) -- (0,-6);
\node at (0,-0.5) {$M_1$};
\draw[thick] (2,-1) -- (2,-6) node [near end, below, circle, draw, fill = lime, opacity = 1] {Anna-1};
\node at (2,-0.5) {$M_2$};
\draw[thick] (4,-1) -- (4,-6) node [near end, below, circle, draw, fill = lime, opacity = 1] {Anna-2};;
\node at (4,-0.5) {$M_3$};
\draw[thick] (6,-1) -- (6,-6);
\node at (6,-0.5) {$M_4$};

\draw[thick] (-2,-2) -- (8,-2);
\node at (-2.5,-2) {$L_1$};
\draw[thick] (-2,-4) -- (8,-4);
\node at (-2.5,-4) {$L_2$};

\node[circle, draw, fill = cyan, opacity = 1] at (1,-3) {Bob-1};
\node[circle, draw, fill = cyan, opacity = 1] at (3,-3) {Bob-2};
\node[circle, draw, fill = cyan, opacity = 1] at (5,-3) {Bob-3};
\end{tikzpicture}
\end{center}

Assume that the observers only see things locally and can record events, but cannot keep track of time when nothing is happening.
Let our team observe the following gradual collisions:

\begin{itemize}
    \item
    In the first case, $m_1$ happens at moment $1$, then $l_1$ happens at moment $2$, then $m_3$ happens at moment $3$ and finally $m_2$ happens in moment $4$.
    
    \medskip
    
    \item
    In the second case, moment $1$ and $2$ change places.
\end{itemize}

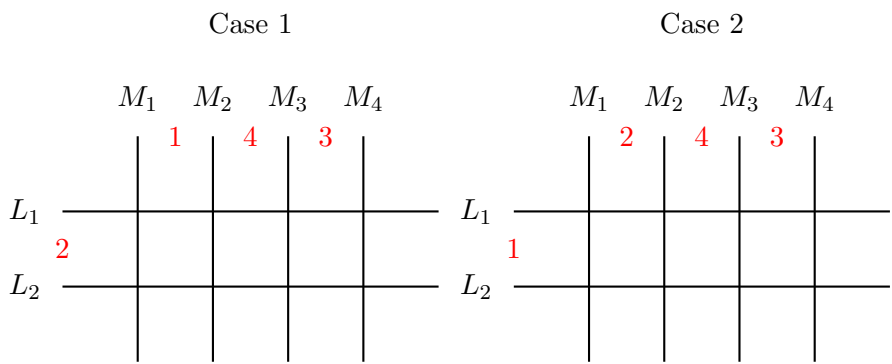
\begin{center}
\begin{tikzpicture}

\draw[thick] (0,0) -- (0,-3);
\node at (0,0.5) {$M_1$};
\draw[thick] (1,0) -- (1,-3);
\node at (1,0.5) {$M_2$};
\draw[thick] (2,0) -- (2,-3);
\node at (2,0.5) {$M_3$};
\draw[thick] (3,0) -- (3,-3);
\node at (3,0.5) {$M_4$};

\draw[thick] (-1,-1) -- (4,-1);
\node at (-1.5,-1) {$L_1$};
\draw[thick] (-1,-2) -- (4,-2);
\node at (-1.5,-2) {$L_2$};

\node[red] at (0.5,0) {1};
\node[red] at (1.5,0) {4};
\node[red] at (2.5,0) {3};

\node[red] at (-1,-1.5) {2};

\node at (1.5, 1.5) {Case 1};

\begin{scope}[shift = {(6,0)}]
\draw[thick] (0,0) -- (0,-3);
\node at (0,0.5) {$M_1$};
\draw[thick] (1,0) -- (1,-3);
\node at (1,0.5) {$M_2$};
\draw[thick] (2,0) -- (2,-3);
\node at (2,0.5) {$M_3$};
\draw[thick] (3,0) -- (3,-3);
\node at (3,0.5) {$M_4$};

\draw[thick] (-1,-1) -- (4,-1);
\node at (-1.5,-1) {$L_1$};
\draw[thick] (-1,-2) -- (4,-2);
\node at (-1.5,-2) {$L_2$};

\node[red] at (0.5,0) {2};
\node[red] at (1.5,0) {4};
\node[red] at (2.5,0) {3};

\node[red] at (-1,-1.5) {1};

\node at (1.5, 1.5) {Case 2};
\end{scope}

\end{tikzpicture} 
\end{center}

\noindent
In each of the cases, let the team meet after the end of time, and discuss.

In case 1, Anna-1 knows that $m_1$ happened earlier then $m_2$: sitting on $M_2$, she saw $M_1$ earlier then she saw $M_3$ and met with Anna-2.
Similarly, Anna-2 knows that $m_3$ happened earlier then $m_2$. Together they are unable to compare $m_1$ with $m_3$, but let Bobs join the discussion.
Bob-1 knows that $m_1$ happened earlier than $l_1$ and Bob-3 knows that $l_1$ happened earlier than $m_3$ (the input from Bob-2 isn't needed).
So together our observers can compare $m_1$ and $m_4$ through $l_1$ and come up with the correct (and full) account: 

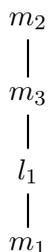
\begin{center}
\begin{tikzpicture}
\node (m2) {$m_2$};
\node (m4) [below of = m2] {$m_3$};
\node (l1) [below of = m4] {$l_1$};
\node (m1) [below of = l1] {$m_1$};

\draw[thick] (m2) -- (m4);
\draw[thick] (m4) -- (l1);
\draw[thick] (l1) -- (m1);

\end{tikzpicture}
\end{center}

In case 2, Annas again agree that $m_1$ and $m_3$ were earlier than $m_2$.
However, at this time, the input from Bobs doesn't help the team to compare $m_1$ and $m_3$: all Bobs simply tell that $l_1$ was earlier than everything else. So the final account is the following preorder:

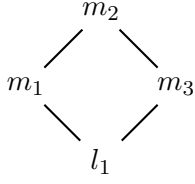
\begin{center}
\begin{tikzpicture}
\node (m2) {$m_2$};
\node (the) [below of = m2] {};
\node (m4) [right of = the] {$m_3$};

\node (m1) [left of = the] {$m_1$};
\node (l1) [below of = the] {$l_1$};

\draw[thick] (m2) -- (m4);
\draw[thick] (m2) -- (m1);

\draw[thick] (m1) -- (l1);
\draw[thick] (l1) -- (m4);

\end{tikzpicture}
\end{center}

This Anna \& Bob metaphor is related to the notion of realizing constrainahedra as Gromov-compactified configuration spaces, as described in \S\ref{ss:gromov}.

\subsection{Associated rectangular bracketings}
Let $a_{ij}$ be the intersection point of $L_i$ and $M_j$, and let $A(m,n)$ be the set of all $a_{ij}$.

\begin{defi}
A rectangular bracket is a subset of $A(m,n)$ of the form
\begin{align}
\label{eq:rect_bracket}
\bigl\{
a_{ij}
\:\big|\:
i_s \leq i \leq i_t,
\:
j_s \leq j \leq j_t
\bigr\}.
\end{align}
\null\hfill$\triangle$
\end{defi}

To a good rectangular preorder $P$, we associate a collection of rectangular brackets $Br(P)$, called a rectangular bracketing.

\begin{defi}
A bracket as in \eqref{eq:rect_bracket}
is added to $Br(P)$ if, according to $P$, all of the collisions $l_i$ for $i_s \leq i < i_t$ and $m_j$ for  $j_s \leq j < j_t$ happened earlier than $l_{i_s-1}$, $i_{i_t}$, $m_{j_s-1}$ and $m_{j_t}$ (if one of those is not defined, it is assumed to happen {\em never}, which is later than anything).
\null\hfill$\triangle$
\end{defi}

\noindent
Informally this means that a bracket embraces items that collide at each moment of time.

\begin{eg}
For the preorder in Example \ref{preord}, its associated rectangular bracketing is this:

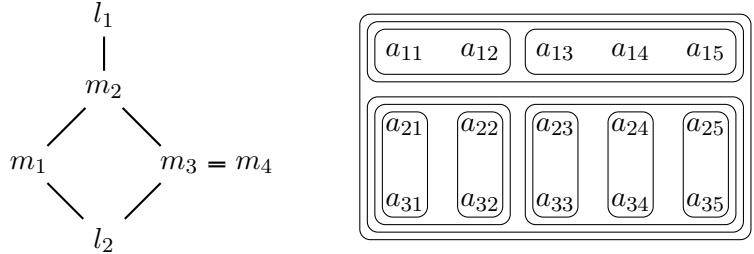
\begin{center}
\begin{tikzpicture}
\node (1) {$l_1$};
\node (2) [below of = 1] {$m_2$};
\node (the) [below of = 2] {};
\node (3) [left of = the] {$m_1$};
\node (4) [right of = the] {$m_3$};
\node (5) [below of = the] {$l_2$};
\node (6) [right of = 4] {$m_4$};

\draw[thick] (1) -- (2);
\draw[thick] (2) -- (3);
\draw[thick] (2) -- (4);
\draw[thick] (4) -- (5);
\draw[thick] (3) -- (5);
\draw[thick, double] (4) -- (6);

\begin{scope}[shift = {(4,-0.5)}]
\node (a11) {$a_{11}$};
\node (a12) [right of=a11] {$a_{12}$};
\node (a13) [right of=a12] {$a_{13}$};
\node (a14) [right of=a13] {$a_{14}$};
\node (a15) [right of=a14] {$a_{15}$};

\node (a21) [below of =a11] {$a_{21}$};
\node (a22) [right of=a21] {$a_{22}$};
\node (a23) [right of=a22] {$a_{23}$};
\node (a24) [right of=a23] {$a_{24}$};
\node (a25) [right of=a24] {$a_{25}$};

\node (a31) [below of =a21] {$a_{31}$};
\node (a32) [right of=a31] {$a_{32}$};
\node (a33) [right of=a32] {$a_{33}$};
\node (a34) [right of=a33] {$a_{34}$};
\node (a35) [right of=a34] {$a_{35}$};

\draw[rounded corners] ($(a21)+(-0.3,0.2)$) rectangle ($(a31)+(0.3,-0.2)$);
\draw[rounded corners] ($(a22)+(-0.3,0.2)$) rectangle ($(a32)+(0.3,-0.2)$);
\draw[rounded corners] ($(a23)+(-0.3,0.2)$) rectangle ($(a33)+(0.3,-0.2)$);
\draw[rounded corners] ($(a24)+(-0.3,0.2)$) rectangle ($(a34)+(0.3,-0.2)$);
\draw[rounded corners] ($(a25)+(-0.3,0.2)$) rectangle ($(a35)+(0.3,-0.2)$);

\draw[rounded corners] ($(a11)+(-0.4,0.3)$) rectangle ($(a12)+(0.4,-0.3)$);
\draw[rounded corners] ($(a21)+(-0.4,0.3)$) rectangle ($(a32)+(0.4,-0.3)$);

\draw[rounded corners] ($(a13)+(-0.4,0.3)$) rectangle ($(a15)+(0.4,-0.3)$);
\draw[rounded corners] ($(a23)+(-0.4,0.3)$) rectangle ($(a35)+(0.4,-0.3)$);

\draw[rounded corners] ($(a11)+(-0.5,0.4)$) rectangle ($(a15)+(0.5,-0.4)$);

\draw[rounded corners] ($(a21)+(-0.5,0.4)$) rectangle ($(a35)+(0.5,-0.4)$);

\draw[rounded corners] ($(a11)+(-0.6,0.5)$) rectangle ($(a35)+(0.6,-0.5)$);

\end{scope}

\end{tikzpicture}
\end{center}

\noindent
For instance, the bracket embracing $a_{21}$, $a_{22}$, $a_{31}$ and $a_{32}$ was added because the collisions $l_2$ and $m_1$ were all earlier than $l_1$ and $m_2$ (and certainly earlier than {\em never}).
\null\hfill$\triangle$
\end{eg}

The assignment $P \mapsto Br(P)$ gives a map from good rectangular preoders to sets of subsets of $A(m,n)$:
\begin{align}
Br: C(m,n) \to 2^{2^{A(m,n)}}
\end{align}

\begin{theo}
$Br$ is injective: the associated rectangular bracketing keeps all the data of a good rectangular preorder.
\end{theo}

\begin{proof}
We explain how the data the data of $P$ can be restored from $Br(P)$. First consider two orthogonal collisions, $l_i$ and $m_j$. Look at the the intersections $a_{i,j}$, $a_{i,j+1}$, $a_{i+1,j}$ and $a_{i+1,j+1}$. There are three possibilities: 
\begin{enumerate}
    \item
    $Br(P)$ has brackets $R_1$, $R_2$ such that $a_{i,j}$ and $a_{i,j+1}$ are in $R_1$ and not in $R_2$, while $a_{i+1,j}$ and $a_{i+1,j+1}$ are in $R_2$ and not in $R_1$. In this case $l_i <_P m_i$.
    
    \item 
    $Br(P)$ has brackets $R_1$, $R_2$ such that $a_{i,j}$ and $a_{i+1,j}$ are in $R_1$ and not in $R_2$, while $a_{i,j+1}$ and $a_{i+1,j+1}$ are in $R_2$ and not in $R_1$.
    In this case $l_i >_P m_i$.
    
    \item
    Every bracket of $Br(P)$ that has $a_{i,j}$ also has $a_{i+1,j+1}$.
    In this case, $l_i \equiv m_j$.
\end{enumerate}

Now consider two parallel collisions, $l_i$ and $l_j$, for $i < j$.
(For $m_i$ and $m_j$, the argument is identical.)
If they were comparable through existence of a link, this data is restored by transitivity.
For comparability through the absence of gaps, check which of the following four possibilities holds:

\begin{enumerate}
    \item
    $Br(P)$ has two brackets $R_1 \subset R_2$ such that $R_1$ contains $a_{i,1}$ and $a_{i+1,1}$ but not $a_{j+1,1}$, and $R_2$ contains $a_{j+1,1}$. In this case $l_i < l_j$.
    
    \item
    $Br(P)$ has two brackets $R_1 \subset R_2$ such that $R_1$ contains $a_{j,1}$ and $a_{j+1,1}$ but not $a_{i,1}$, and $R_2$ contains $a_{i,1}$. In this case $l_i > l_j$.
    
    \item
    $Br(P)$ has a bracket that contains $a_{i,1}$ and $a_{j+1,1}$, and for every $R' \subset R$ the same holds. In this case $l_i \equiv l_j$.
    
    \item
    None of the above; in this case either they $l_i$ and $l_j$ are comparable via an orthogonal link, or incomparable. 
\end{enumerate}
\end{proof}

\noindent
This shows that rectangular bracketings can be used as a convenient visualization of preorders.

\begin{rem}
When looking at a collection of rectangular brackets, it might be not immediately obvious whether this collection is associated to a preorder.
For example, the reader might check that the following collection is not.

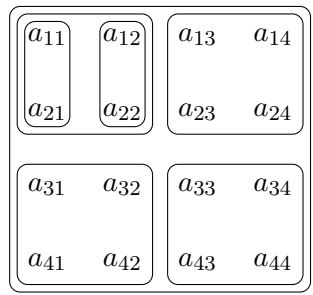
\begin{center}
\begin{tikzpicture}

\node (a11) {$a_{11}$};
\node (a12) [right of=a11] {$a_{12}$};
\node (a13) [right of=a12] {$a_{13}$};
\node (a14) [right of=a13] {$a_{14}$};

\node (a21) [below of =a11] {$a_{21}$};
\node (a22) [right of=a21] {$a_{22}$};
\node (a23) [right of=a22] {$a_{23}$};
\node (a24) [right of=a23] {$a_{24}$};

\node (a31) [below of =a21] {$a_{31}$};
\node (a32) [right of=a31] {$a_{32}$};
\node (a33) [right of=a32] {$a_{33}$};
\node (a34) [right of=a33] {$a_{34}$};

\node (a41) [below of =a31] {$a_{41}$};
\node (a42) [right of=a41] {$a_{42}$};
\node (a43) [right of=a42] {$a_{43}$};
\node (a44) [right of=a43] {$a_{44}$};

\draw[rounded corners] ($(a11)+(-0.3,0.2)$) rectangle ($(a21)+(0.3,-0.2)$);
\draw[rounded corners] ($(a12)+(-0.3,0.2)$) rectangle ($(a22)+(0.3,-0.2)$);

\draw[rounded corners] ($(a11)+(-0.4,0.3)$) rectangle ($(a22)+(0.4,-0.3)$);
\draw[rounded corners] ($(a13)+(-0.4,0.3)$) rectangle ($(a24)+(0.4,-0.3)$);
\draw[rounded corners] ($(a31)+(-0.4,0.3)$) rectangle ($(a42)+(0.4,-0.3)$);
\draw[rounded corners] ($(a33)+(-0.4,0.3)$) rectangle ($(a44)+(0.4,-0.3)$);

\draw[rounded corners] ($(a11)+(-0.5,0.4)$) rectangle ($(a44)+(0.5,-0.4)$);

\end{tikzpicture}
\end{center}
\null\hfill$\triangle$
\end{rem}

\begin{eg}
Below is $C(2,3)$ in terms of rectangular bracketings, overlaid with a hexagon.

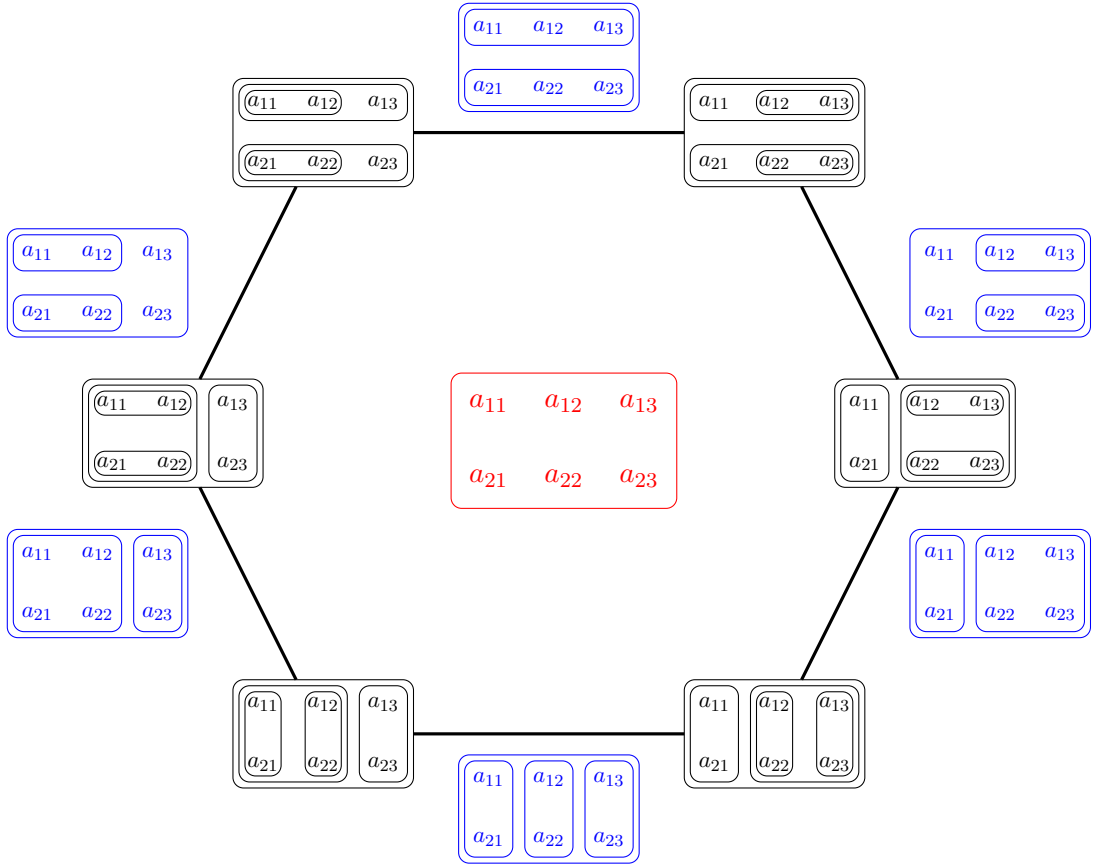
\begin{center}
\begin{tikzpicture}

\begin{scope}[scale=0.8, every node/.style={transform shape},local bounding box=one]
\node (a11) {$a_{11}$};
\node (a12) [right of=a11] {$a_{12}$};
\node (a13) [right of=a12] {$a_{13}$};

\node (a21) [below of =a11] {$a_{21}$};
\node (a22) [right of=a21] {$a_{22}$};
\node (a23) [right of=a22] {$a_{23}$};

\draw[rounded corners] ($(a11)+(-0.3,0.2)$) rectangle ($(a12)+(0.3,-0.2)$);

\draw[rounded corners] ($(a21)+(-0.3,0.2)$) rectangle ($(a22)+(0.3,-0.2)$);

\draw[rounded corners] ($(a11)+(-0.4,0.3)$) rectangle ($(a13)+(0.4,-0.3)$);

\draw[rounded corners] ($(a21)+(-0.4,0.3)$) rectangle ($(a23)+(0.4,-0.3)$);

\draw[rounded corners] ($(a11)+(-0.5,0.4)$) rectangle ($(a23)+(0.5,-0.4)$);
\end{scope}

\begin{scope}[color = blue, shift = {(-3,-2)}, scale=0.8, every node/.style={transform shape}]
\node (a11) {$a_{11}$};
\node (a12) [right of=a11] {$a_{12}$};
\node (a13) [right of=a12] {$a_{13}$};

\node (a21) [below of =a11] {$a_{21}$};
\node (a22) [right of=a21] {$a_{22}$};
\node (a23) [right of=a22] {$a_{23}$};

\draw[rounded corners] ($(a11)+(-0.4,0.3)$) rectangle ($(a12)+(0.4,-0.3)$);

\draw[rounded corners] ($(a21)+(-0.4,0.3)$) rectangle ($(a22)+(0.4,-0.3)$);

\draw[rounded corners] ($(a11)+(-0.5,0.4)$) rectangle ($(a23)+(0.5,-0.4)$);
\end{scope}

\begin{scope}[color = red, shift = {(3,-4)}, every node/.style={transform shape}]
\node (a11) {$a_{11}$};
\node (a12) [right of=a11] {$a_{12}$};
\node (a13) [right of=a12] {$a_{13}$};

\node (a21) [below of =a11] {$a_{21}$};
\node (a22) [right of=a21] {$a_{22}$};
\node (a23) [right of=a22] {$a_{23}$};

\draw[rounded corners] ($(a11)+(-0.5,0.4)$) rectangle ($(a23)+(0.5,-0.4)$);
\end{scope}

\begin{scope}[color = blue, shift = {(3,1)}, scale=0.8, every node/.style={transform shape}]
\node (a11) {$a_{11}$};
\node (a12) [right of=a11] {$a_{12}$};
\node (a13) [right of=a12] {$a_{13}$};

\node (a21) [below of =a11] {$a_{21}$};
\node (a22) [right of=a21] {$a_{22}$};
\node (a23) [right of=a22] {$a_{23}$};

\draw[rounded corners] ($(a11)+(-0.4,0.3)$) rectangle ($(a13)+(0.4,-0.3)$);

\draw[rounded corners] ($(a21)+(-0.4,0.3)$) rectangle ($(a23)+(0.4,-0.3)$);

\draw[rounded corners] ($(a11)+(-0.5,0.4)$) rectangle ($(a23)+(0.5,-0.4)$);
\end{scope}

\begin{scope}[shift = {(6,0)}, scale=0.8, every node/.style={transform shape},local bounding box = two]
\node (a11) {$a_{11}$};
\node (a12) [right of=a11] {$a_{12}$};
\node (a13) [right of=a12] {$a_{13}$};

\node (a21) [below of =a11] {$a_{21}$};
\node (a22) [right of=a21] {$a_{22}$};
\node (a23) [right of=a22] {$a_{23}$};

\draw[rounded corners] ($(a12)+(-0.3,0.2)$) rectangle ($(a13)+(0.3,-0.2)$);

\draw[rounded corners] ($(a22)+(-0.3,0.2)$) rectangle ($(a23)+(0.3,-0.2)$);

\draw[rounded corners] ($(a11)+(-0.4,0.3)$) rectangle ($(a13)+(0.4,-0.3)$);

\draw[rounded corners] ($(a21)+(-0.4,0.3)$) rectangle ($(a23)+(0.4,-0.3)$);

\draw[rounded corners] ($(a11)+(-0.5,0.4)$) rectangle ($(a23)+(0.5,-0.4)$);
\end{scope}

\begin{scope}[shift = {(-2,-4)},scale=0.8, every node/.style={transform shape},local bounding box = six]
\node (a11) {$a_{11}$};
\node (a12) [right of=a11] {$a_{12}$};
\node (a13) [right of=a12] {$a_{13}$};

\node (a21) [below of =a11] {$a_{21}$};
\node (a22) [right of=a21] {$a_{22}$};
\node (a23) [right of=a22] {$a_{23}$};

\draw[rounded corners] ($(a11)+(-0.3,0.2)$) rectangle ($(a12)+(0.3,-0.2)$);

\draw[rounded corners] ($(a21)+(-0.3,0.2)$) rectangle ($(a22)+(0.3,-0.2)$);

\draw[rounded corners] ($(a11)+(-0.4,0.3)$) rectangle ($(a22)+(0.4,-0.3)$);

\draw[rounded corners] ($(a13)+(-0.4,0.3)$) rectangle ($(a23)+(0.4,-0.3)$);

\draw[rounded corners] ($(a11)+(-0.5,0.4)$) rectangle ($(a23)+(0.5,-0.4)$);
\end{scope}

\begin{scope}[shift = {(0,-8)},scale=0.8, every node/.style={transform shape},local bounding box = five]
\node (a11) {$a_{11}$};
\node (a12) [right of=a11] {$a_{12}$};
\node (a13) [right of=a12] {$a_{13}$};

\node (a21) [below of =a11] {$a_{21}$};
\node (a22) [right of=a21] {$a_{22}$};
\node (a23) [right of=a22] {$a_{23}$};

\draw[rounded corners] ($(a11)+(-0.3,0.2)$) rectangle ($(a21)+(0.3,-0.2)$);

\draw[rounded corners] ($(a12)+(-0.3,0.2)$) rectangle ($(a22)+(0.3,-0.2)$);

\draw[rounded corners] ($(a11)+(-0.4,0.3)$) rectangle ($(a22)+(0.4,-0.3)$);

\draw[rounded corners] ($(a13)+(-0.4,0.3)$) rectangle ($(a23)+(0.4,-0.3)$);

\draw[rounded corners] ($(a11)+(-0.5,0.4)$) rectangle ($(a23)+(0.5,-0.4)$);
\end{scope}

\begin{scope}[shift = {(8,-4)}, scale=0.8, every node/.style={transform shape},local bounding box = three]
\node (a11) {$a_{11}$};
\node (a12) [right of=a11] {$a_{12}$};
\node (a13) [right of=a12] {$a_{13}$};

\node (a21) [below of =a11] {$a_{21}$};
\node (a22) [right of=a21] {$a_{22}$};
\node (a23) [right of=a22] {$a_{23}$};

\draw[rounded corners] ($(a12)+(-0.3,0.2)$) rectangle ($(a13)+(0.3,-0.2)$);

\draw[rounded corners] ($(a22)+(-0.3,0.2)$) rectangle ($(a23)+(0.3,-0.2)$);

\draw[rounded corners] ($(a12)+(-0.4,0.3)$) rectangle ($(a23)+(0.4,-0.3)$);

\draw[rounded corners] ($(a11)+(-0.4,0.3)$) rectangle ($(a21)+(0.4,-0.3)$);

\draw[rounded corners] ($(a11)+(-0.5,0.4)$) rectangle ($(a23)+(0.5,-0.4)$);
\end{scope}

\begin{scope}[color = blue, shift = {(9,-2)}, scale=0.8, every node/.style={transform shape},local bounding box]
\node (a11) {$a_{11}$};
\node (a12) [right of=a11] {$a_{12}$};
\node (a13) [right of=a12] {$a_{13}$};

\node (a21) [below of =a11] {$a_{21}$};
\node (a22) [right of=a21] {$a_{22}$};
\node (a23) [right of=a22] {$a_{23}$};

\draw[rounded corners] ($(a12)+(-0.4,0.3)$) rectangle ($(a13)+(0.4,-0.3)$);

\draw[rounded corners] ($(a22)+(-0.4,0.3)$) rectangle ($(a23)+(0.4,-0.3)$);

\draw[rounded corners] ($(a11)+(-0.5,0.4)$) rectangle ($(a23)+(0.5,-0.4)$);
\end{scope}

\begin{scope}[color = blue, shift = {(9,-6)}, scale=0.8, every node/.style={transform shape}]
\node (a11) {$a_{11}$};
\node (a12) [right of=a11] {$a_{12}$};
\node (a13) [right of=a12] {$a_{13}$};

\node (a21) [below of =a11] {$a_{21}$};
\node (a22) [right of=a21] {$a_{22}$};
\node (a23) [right of=a22] {$a_{23}$};

\draw[rounded corners] ($(a12)+(-0.4,0.3)$) rectangle ($(a23)+(0.4,-0.3)$);

\draw[rounded corners] ($(a11)+(-0.4,0.3)$) rectangle ($(a21)+(0.4,-0.3)$);

\draw[rounded corners] ($(a11)+(-0.5,0.4)$) rectangle ($(a23)+(0.5,-0.4)$);
\end{scope}

\begin{scope}[shift = {(6,-8)}, scale=0.8, every node/.style={transform shape},local bounding box = four]
\node (a11) {$a_{11}$};
\node (a12) [right of=a11] {$a_{12}$};
\node (a13) [right of=a12] {$a_{13}$};

\node (a21) [below of =a11] {$a_{21}$};
\node (a22) [right of=a21] {$a_{22}$};
\node (a23) [right of=a22] {$a_{23}$};

\draw[rounded corners] ($(a12)+(-0.3,0.2)$) rectangle ($(a22)+(0.3,-0.2)$);

\draw[rounded corners] ($(a13)+(-0.3,0.2)$) rectangle ($(a23)+(0.3,-0.2)$);

\draw[rounded corners] ($(a12)+(-0.4,0.3)$) rectangle ($(a23)+(0.4,-0.3)$);

\draw[rounded corners] ($(a11)+(-0.4,0.3)$) rectangle ($(a21)+(0.4,-0.3)$);

\draw[rounded corners] ($(a11)+(-0.5,0.4)$) rectangle ($(a23)+(0.5,-0.4)$);
\end{scope}

\begin{scope}[color = blue, shift = {(3,-9)}, scale=0.8, every node/.style={transform shape}]
\node (a11) {$a_{11}$};
\node (a12) [right of=a11] {$a_{12}$};
\node (a13) [right of=a12] {$a_{13}$};

\node (a21) [below of =a11] {$a_{21}$};
\node (a22) [right of=a21] {$a_{22}$};
\node (a23) [right of=a22] {$a_{23}$};

\draw[rounded corners] ($(a12)+(-0.4,0.3)$) rectangle ($(a22)+(0.4,-0.3)$);

\draw[rounded corners] ($(a13)+(-0.4,0.3)$) rectangle ($(a23)+(0.4,-0.3)$);

\draw[rounded corners] ($(a11)+(-0.4,0.3)$) rectangle ($(a21)+(0.4,-0.3)$);

\draw[rounded corners] ($(a11)+(-0.5,0.4)$) rectangle ($(a23)+(0.5,-0.4)$);
\end{scope}

\begin{scope}[color = blue, shift = {(-3,-6)}, scale=0.8, every node/.style={transform shape}]
\node (a11) {$a_{11}$};
\node (a12) [right of=a11] {$a_{12}$};
\node (a13) [right of=a12] {$a_{13}$};

\node (a21) [below of =a11] {$a_{21}$};
\node (a22) [right of=a21] {$a_{22}$};
\node (a23) [right of=a22] {$a_{23}$};

\draw[rounded corners] ($(a13)+(-0.4,0.3)$) rectangle ($(a23)+(0.4,-0.3)$);

\draw[rounded corners] ($(a11)+(-0.4,0.3)$) rectangle ($(a22)+(0.4,-0.3)$);

\draw[rounded corners] ($(a11)+(-0.5,0.4)$) rectangle ($(a23)+(0.5,-0.4)$);
\end{scope}

\draw[very thick] (one) -- (two) -- (three) -- (four) -- (five) -- (six) -- (one);
\end{tikzpicture}
\end{center}
\null\hfill$\triangle$
\end{eg}

\section{Constrainahedra are lattices}
\label{sec:lattice}

In this section, we will show that $C(m,n)$, with a formal minimal element adjoined, is a lattice.
Using this property, we will go on to establish a stronger property in \S\ref{sec:polytopes}: $C(m,n)$ is the face lattice of an embedded polytope.

We begin with a technical lemma about parallel comparisons.

\begin{lem}
\label{lem:tech}
For a parallel comparison $m_i \leq m_j$ in some good rectangular preorder $P$, we have $m_s \leq m_j$ for every $s \in [i,j]$.
\end{lem}

\begin{proof}
We need to show that $m_s$ cannot be incomparable to $m_j$. Assume the contrary; then there exists a gap $m_k$ with $k \in [s,j]$ and $m_k > m_j$.
But $[s,j] \subset [i,j]$, so $m_k$ is also a gap between $m_i$ and $m_j$, which contradicts their comparability.
\end{proof}

Using Lemma~\ref{lem:tech}, we now define an operation that will turn out to be the join in $C(m,n) \cup \{-1\}$.

\begin{lem}
\label{join}
Let $P$ and $Q$ be two good rectangular preorders, viewed as subsets in $\Coll(m,n) \times \Coll(m,n)$.
Let  $\overline{P \cup Q}$ denote the transitive closure of $P \cup Q$.
Then $\overline{P \cup Q}$  is also a good rectangular preorder.
\end{lem}

\begin{proof}
{\sc (Orthogonal comparability)} is satisfied trivially: orthogonal collisions are comparable in each of the preorders, so certainly in their union and in its transitive closure.

{\sc (Parallel comparability)} requires more work.
First assume that for two parallel collisions there exists a $\PQ$-link.
Then they are comparable by transitivity of $\PQ$.
Now assume that for two parallel collisions (without loss of generality call them $m_i$ and $m_j$) there are no $\PQ$-gaps between them.

If this is due to absence of gaps in $P$ or in $Q$, then comparability follows from the axioms on $P$ or $Q$.
So assume that both $P$ and $Q$ have a gap between $m_i$ and $m_j$.
Let $m_s$ be such a gap in $P$: $m_i <_P m_s >_P m_j$. The fact that this gap disappears in $\PQ$ means that one of strict inequalities becomes an equivalence: either $m_i \geq_{\PQ} m_s$ or $m_j \geq_{\PQ} m_s$.
In the first case, we have $m_i \geq_{\PQ} m_s >_P m_j$, and in the second case we have $m_i <_P m_s \leq_{\PQ} m_j$ --- in both cases this gives a $\PQ$-comparability, because $\PQ$ is transitive and refines $P$.

In the other direction: assume that for two parallel collisions $m_i$ and $m_j$ that there are no $\PQ$-links and a $\PQ$-gap $m_s$.
We must show that they cannot be $\PQ$-comparable.

Assume the contrary --- that they are $\PQ$-comparable, which means that there exists a chain $m_i = x_0 \leq x_1 \leq \ldots \leq x_p = m_j$, where each of inequalities $x_a \leq x_{a+1}$ holds either in $P$ or in $Q$.
Firstly we notice that none of $x_k$ can be orthogonal to $m_i$ and $m_j$: this would give an orthogonal link.
This means that each of $x_a \leq x_{a+1}$ is a parallel comparison of $m$'s which happens due to absence of gaps. We now recall our $m_s$, a $\PQ$-gap.
Consider the step $x_a \leq_P x_{a+1}$ over $m_s$.
Then we have $x_{a+1} \geq_P m_s$ by Lemma \ref{lem:tech}, and $x_{a+1} \geq_P m_s >_{\PQ} m_j$ contradicts $x_{a+1} \leq_{\PQ} m_j$.
\end{proof}

Next, we define an operation that will turn out to be the meet in $C(m,n) \sqcup \{-1\}$.

\begin{lem}
\label{meet}
Let $P$ and $Q$ be two good rectangular preorders.
Assume that $P \cap Q$ satisfies {(\sc orthogonal comparability)}, and absence of gaps implies parallel comparability.
Then $P \cap Q$ is a good rectangular preorder.
\end{lem}

\begin{proof}
We are left to check the converse: that parallel comparability implies that there is either an orthogonal link or no gaps.
Assume the contrary:
that there are two parallel collisions $m_i \leq_{P \cap Q} m_j$ such that there is a gap and no link between them.
We notice that absence of link in $P \cap Q$ means that there was no link both in $P$ and in $Q$: $P \cap Q$ can only satisfy {\sc (orthogonal comparability)} is all orthogonal comparisons coincide in $P$ and in $Q$.
So $m_i \leq_P m_j$ means that there is no gap in $P$ and $m_i \leq_Q m_j$ means that there is no gap in $Q$.
Now consider $m_s$ which is a $P \cap Q$-gap between $m_i$ and $m_j$.
Since $m_s \geq m_j$ holds in $P \cap Q$, it must also hold both in $P$ and in $Q$, so the only chance for $m_s$ not to be a $P$-gap is $m_s \equiv_P m_j$, and the only chance for $m_s$ not to be a $Q$-gap is $m_s \equiv_Q m_j$.
But then $m_s \equiv_{P \cap Q} m_j$, which contradicts its being a gap.
\end{proof}

We are now ready to prove the main result of this section.

\begin{theo}
\label{lattice}
The poset $C(m,n) \sqcup \{-1 \}$ is a lattice.
\end{theo}
\begin{proof}
Let $P$ and $Q$ be two good rectangular preorders.
The existence of their join in $C(m,n)$ follows from Lemma \ref{join}: $\PQ$ is the join of $P$ and $Q$ among all the rectangular posets, so certainly among good ones.

For their meet, check if $P \cap Q$ (the meet of $P$ and $Q$ among all preorders) is good.
If it is, then this is the meet.
Otherwise the meet is $-1$. For this to be true, we need to verify that if $P \cap Q$ is not good, then there is no good preorder refined by it.

Assume the contrary. By Lemma \ref{meet}, failure to be good means either lack of orthogonal comparability or lack of parallel comparability for a pair with no gap.
Lack of orthogonal comparability cannot be rectified by further coarsening. So let $m_i \#_{P \cap Q} m_j$ be an incomparable parallel pair with no $P\cap Q$-gap.
Rectification by coarsening would mean creating this gap, by replacing one or two equivalences by strict inequalities, $m_i \equiv m_s \equiv m_j$ $\implies$ $m_i < m_s > m_j$ or $m_i \equiv m_s > m_j$ $\implies$ $m_i < m_s > m_j$ --- but in both cases we see that $m_j$ and $m_j$ are already comparable in $P \cap Q$.
\end{proof}

The reward for our labour is that we are now in a position to use the following lemma, which will be a key ingredient in our construction of polytopal realizations of the constrainahedra.

\begin{lem}[Exercise 3.27, \cite{stanley_vol_1}]
Let $L$ be a finite lattice, and define the subposet $\Irr(L)$ of irreducibles of $L$ by
\begin{align}
\Irr(L)
\coloneqq
\bigl\{
x \in L
\:\big|\:
x
\text{ is join-irreducible or meet-irreducible (or both)}
\bigr\}
\end{align}
Then $L$ can be uniquely recovered from the poset $\Irr(L)$.
\null\hfill$\square$
\end{lem} 

In the case of $C(m,n)$, the irreducibles are easy to describe. 

\begin{lem}
In $C(m,n)$, the vertices (join-irreducibles) are those preorders in which equivalence implies equality.
In particular, the number of equivalence classes is $n+m-2$, and preorders are actually posets. 
\end{lem}

\begin{proof}
Let $P$ be such a preorder.
We show that it does not refine any other preorder different from itself.
Assume the contrary: $P' \subsetneq P$.
Then there exists a comparison $a < b$ present in $P$ but not in $P'$. It can only be a parallel comparison, say $m_i < m_j$, with either no gap or a link.
But removing comparisons cannot cancel a link (orthogonal comparisons cannot be removed), and removing comparisons can only create a gap if there was a nontrivial equivalence --- which is not so in the case of $P$ being a poset. Thus indeed $P = \overline{Q \cup Q'}$ implies that either $Q$ or $Q'$ has to coincide with $P$.
\end{proof}

\begin{lem}
In $C(m,n)$, the facets (meet-irreducibles) are those preorders where the number of equivalence classes is $2$.
\end{lem}

\begin{proof}
Let $P$ be such a preorder.
The only possible refinement joins the two equivalence classes in one.
Thus $P = Q \cap Q'$ trivially implies that either $Q$ or $Q'$ has to coincide with $P$.
\end{proof}

Additionally we will speak of {\em edges}.

\begin{defi}
An edge preorder is a preorder where the number of equivalence classes is $n+m-1$, meaning that there is exactly one class that consists of 2 simultaneous collisions, and every other class consists of 1 collision.
\null\hfill$\triangle$
\end{defi} 

Right now this is just an abstract definition.
Let us prove some edgy properties of edges. 

\begin{lem}
Let $P$ be a vertex preorder in $C(m,n)$.
The edge preorders containing $V$ are in bijection with the edges in the Hasse diagram of $P$.
\end{lem}

\begin{proof}
It is obvious that any edge preorder refining $P$ replaces $a<b$ with $a \equiv b$ for some covering relation.
Now let $e$ be an edge in the Hasse diagram of $P$, corresponding to the covering relation $a<b$ for some $a, b \in \Coll$.
We define the edge preorder $E(e)$ to be the transitive closure of $P \cup \{ b \leq a \}$. What we need to check is that $E(e)$ satisfies the axioms of a good rectangular preorder, and is an edge.

{\sc (Orthogonal comparability)} is satisfied trivially: orthogonal collisions were already comparable in P.
Parallel collisions with an orthogonal link between them are comparable by transitivity.
Let $m_i$ and $m_j$ be parallel collisions with no gap between them.
Assume there was a gap $m_s$ in $P$ that disappeared in $E(e)$ --- this means that $m_s$ became equivalent say to $m_j$, thus yielding $m_i < m_s \equiv m_j$ which is a comparison by transitivity.
In the other direction, let $m_i$ and $m_j$ be $E(e)$-comparable collisions.
If they are also $P$-comparable then there is either a link or no gap in $P$ and thus also in $E(e)$.
So assume that $m_i \leq_P b \leq a \leq_P m_j$.
If any of $a$, $b$ is orthogonal to $m_i$ and $m_j$, then it gives a link; assume not.
But then there is no gap between $m_i$ and $b$, between $b$ and $a$, and between $a$ and $m_j$ --- thus no gap between $m_i$ and $m_j$.
This establishes {\sc (Parallel comparability)}.

Let us now check that $E(e)$ is an edge, i.e.\ that $\{a,b\}$ is the single equivalence class of size 2.
Assume the contrary. Let $c \equiv d$ be a new equivalence, with $c \leq_P d$ being a comparison in P, and and $d \leq c$ being added in the transitive closure of $P \cup \{ a>b \}$.
Then $c \geq d$ can be expressed as $d \leq_P b \leq a \leq_P  c$.
But this allows expressing $a \leq_P b$ as $a \leq_P c \leq_P d \leq_P b$ thus contradicting to $a<b$ being an edge in the Hasse diagram.
\end{proof}

\begin{lem}
In $C(m,n)$, every edge preorder $V$ contains exactly two vertex preorders. 
\end{lem}

\begin{proof}
Let $a \equiv b$ be the single two-element class of $E$. We need to show that there exists a good rectangular preorder $E'$ refined by $E$ with  $a < b$, and that any two such preorders coincide.
Both goals are achieved simultaneously by checking that minimal necessary modifications already result in a good rectangular preorder.

We consider two cases.
First assume that $a$ and $b$ are orthogonal, say $a = m_i$ and $b = l_j$.
Then any $E'$ should {\em not} have comparisons $ m_{j'} < l_i < m_j $ where the $l_i$ is the single link between $m_{j'}$ and $m_j$, and there exists a gap $m_s$ for some $s \in [j',j]$.
And, similarly, $E'$ should {\em not} have comparisons $ l_i < m_j < l_{i'}$
where the $m_j$ is the single link between $l_{i}$ and $l_{i'}$, and there exists a gap $l_t$ for some $t \in [i,i']$.
We find out that removing these illegitimate comparisons already results in a good rectangular preorder. Indeed, removing this comparisons does not affect transitivity or {\sc (Orthogonal comparability)}; for {\sc (Parallel comparability)}, we observe that no new gaps are created this way.

Now assume $a = m_i$ and $b = m_j$, with $m_s < m_i$, $m_s< m_j$ for $s \in [i,j]$.
Then, after replacing $m_i \equiv m_j$ with $m_i < m_j$, we can get that $m_j$ becomes a gap between $m_i$ and $m_{j'}$ with $j \in [i,j']$, so whenever there is no orthogonal link, $E'$ should not have these comparisons $m_i > m_{j'}$.
We again find out that removing them results in a good rectangular preorder.
Transitivity and {\sc (Orthogonal comparability)} are again trivially not affected, and the removals allow for {\sc (Parallel comparability)}.
\end{proof}

Slightly informally, the two lemmas above mean that passing from a vertex to another vertex along an edge is achieved by swapping the respective edge in the Hasse diagram of the first vertex preorder, and adjusting the result accordingly (throwing away edges for illegitimate comparisons, drawing edges for legitimate comparisons that were not previously covering). \\

For the next section, we need some understanding of edge-connectedness.

\begin{lem}
\label{connect}
Any two vertices of $C(m,n)$ are connected by a sequence of edges. 
\end{lem}
\begin{proof}
Let $v$ and $w$ be two vertices that we want to connect. We begin the inductive procedure as follows.
Let $a \in \Coll(m,n)$ be the maximal element in $P_w$. We find the same element in $P_v$ and swap edges along (any) sequence connecting $a$ to the top.
Thus we obtain a preorder where $a$ is maximal (as in $P_w$).
Now let us say that an element $x \in \Coll(m,n)$ is \emph{$w$-placed} if all its upgoing Hasse edges are the same as in $P_w$ (thus the base of induction consisted of $w$-placing $a$).
Let $P'_v$ denote the modified preorder.
Let $s \in \Coll(m,n)$ be such that all its upper $w$-neighbours are $w$-placed in $P'_v$. Let $t$ be one of those $w$-neighbours which is not a neighbour in $P'_v$. Choose the shortest path connecting $s$ to $t$ in $P'_v$, and swap edges along that.
This procedure does not affect the $w$-placed part of the diagram.
Therefore we can continue until all the collisions are $w$-placed.
\end{proof}

\begin{lem}
\label{connect2}
Any two vertices inside one facet of $C(m,n)$ are connected by a sequence of edges within that facet.
\end{lem}

\begin{proof}
A facet $F$ is given by two equivalence classes, $C_1 <_F C_2$.
A vertex $v$ belongs to $F$ if and only if for every $x \in C_1$ and $y \in C_2$ we have $x <_v y$.
For two vertices satisfying this, the algorithm explained in the proof of Lemma \ref{connect} never requires to swap collisions from $C_1$ and $C_2$, so the edge sequence stays within $F$.
\end{proof}

\section{A convex hull realization of \texorpdfstring{$C(m,n)$}{C(m,n)}}
\label{sec:polytopes}

We now provide an explicit polytopal realization of constrainahedra, by giving formulas for vertex coordinates.
These formulas emerged from ongoing joint work with Spencer Backman.

Fix a vertex $v \in C(m,n)$.
To this vertex, we will associate:
\begin{itemize}
    \item 
    {\em horizontal coordinates} $y_1$, $\ldots$, $y_{n-1}$ (with $y_i$ corresponding to the collision $l_i$), and
    
    \item
    {\em vertical coordinates} $x_1$, $\ldots$, $x_{m-1}$ (with $x_j$ corresponding to the collision $m_j$).
\end{itemize}
Every coordinate (no matter horizontal or vertical) will be obtained as a product $W_1W_2T$ of three nonnegative integers, where $W_1$ is the {\em first weight}, $W_2$ is the {\em second weight}, and $T$ is the {\em thickness}.

To give the definitions of these numbers, we need some additional terminology.

\begin{defi}
A \emph{partial binary bracketing (PBB)} is an arrangement of brackets obtained from a binary bracketing by removing some brackets in such a way that every remaining bracket is binary.
\null\hfill$\triangle$
\end{defi}

\noindent
For example, $ab(cd)$ and $((ab)(cd))$ are PBBs, and $(ab(cd))$ is not a PBB.

\begin{defi}
The \emph{thickness} of a PBB is the number of pairs $(a_i, a_j)$ such that there exists a bracket embracing (at any depth) both $a_i$ and $a_j$, plus 1.
\null\hfill$\triangle$
\end{defi}

\noindent
In the examples above,  $ab(cd)$ has weight 2 with the only such pair being $(c,d)$, and $((ab)(cd))$ has weight $7 = 1 + \binom{4}{2}$, because every pair of letters is embraced by some bracket.

Next, we define agglomerations of lines.

\begin{defi}
For a collision $l_i$, the \emph{agglomeration $|\!\operatorname{Ag}_{l_i}(L_i)|$ of $L_i$} is the set of lines that have collided with $L_i$ earlier than $l_i$, and the \emph{agglomeration $|\!\operatorname{Ag}_{l_i}(L_{i+1})|$  of $L_{i+1}$} consists of lines that have collided with $L_{i+1}$ earlier than $l_i$.
Agglomerations for collisions $m_j$ are defined similarly.
\null\hfill$\triangle$
\end{defi}

We are now ready to give the formulas for the vertex coordinates.

\begin{defi}
We define the $i$-th horizontal coordinate $y_i$, associated to $v \in C(m,n)$ a vertex and corresponding to the collision $l_i$, by $y_i \coloneqq W_1W_2T$, where:
\begin{itemize}
    \item
    $W_1 \coloneqq |\!\operatorname{Ag}_{l_i}(L_i)|$,
    
    \smallskip
    
    \item
    $W_2 \coloneqq |\!\operatorname{Ag}_{l_i}(L_{i+1})|$, and
    
    \smallskip
    
    \item
    $T$ is the thickness of the PBB whose elements are lines $M_j$ and whose brackets come from vertical collisions that happened before $l_i$.
\end{itemize}

\noindent
We define $x_j$ in a completely analogous way.
That is, if $x_j$ corresponds to the collision $m_j$, then we set $x_j \coloneqq W_1W_2T$, where:
\begin{itemize}
    \item
    $W_1 \coloneqq |\!\operatorname{Ag}_{m_j}(M_j)|$,
    
    \smallskip
    
    \item
    $W_2 \coloneqq |\!\operatorname{Ag}_{m_j}(M_{j+1})|$, and
    
    \smallskip
    
    \item
    $T$ is the thickness of the PBB whose elements are lines $L_i$ and whose brackets come from horizontal collisions that happened before $m_j$.
\end{itemize}

\noindent
We will denote the point with coordinates $x_j$ and $y_i$ by $(\bx,\by)_v$, where the subscript indicates the dependence on the vertex $v \in C(m,n)$.
\null\hfill$\triangle$
\end{defi}

\begin{eg}
Consider the following vertex:

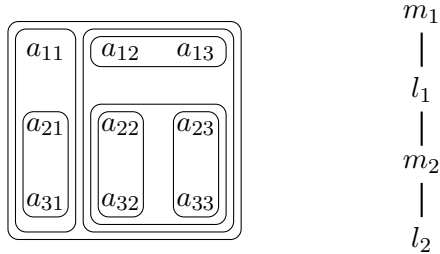
\begin{center}
\begin{tikzpicture}

\node (a11) {$a_{11}$};
\node (a12) [right of=a11] {$a_{12}$};
\node (a13) [right of=a12] {$a_{13}$};

\node (a21) [below of =a11] {$a_{21}$};
\node (a22) [right of=a21] {$a_{22}$};
\node (a23) [right of=a22] {$a_{23}$};

\node (a31) [below of =a21] {$a_{31}$};
\node (a32) [right of=a31] {$a_{32}$};
\node (a33) [right of=a32] {$a_{33}$};

\draw[rounded corners] ($(a21)+(-0.3,0.2)$) rectangle ($(a31)+(0.3,-0.2)$);
\draw[rounded corners] ($(a22)+(-0.3,0.2)$) rectangle ($(a32)+(0.3,-0.2)$);
\draw[rounded corners] ($(a23)+(-0.3,0.2)$) rectangle ($(a33)+(0.3,-0.2)$);

\draw[rounded corners] ($(a22)+(-0.4,0.3)$) rectangle ($(a33)+(0.4,-0.3)$);
\draw[rounded corners] ($(a12)+(-0.4,0.2)$) rectangle ($(a13)+(0.4,-0.2)$);

\draw[rounded corners] ($(a11)+(-0.4,0.3)$) rectangle ($(a31)+(0.4,-0.4)$);
\draw[rounded corners] ($(a12)+(-0.5,0.3)$) rectangle ($(a33)+(0.5,-0.4)$);

\draw[rounded corners] ($(a11)+(-0.5,0.4)$) rectangle ($(a33)+(0.6,-0.5)$);

\begin{scope}[shift = {(5,0.5)}]
\node (1) {$m_1$};
\node (2) [below of = 1] {$l_1$};
\node (3) [below of = 2] {$m_2$};
\node (4) [below of = 3] {$l_2$};

\draw[thick] (1) -- (2) -- (3) -- (4);
\end{scope}

\end{tikzpicture}
\end{center}

\noindent
We compute the coordinates by the procedure explained above.

\begin{enumerate}
    \item
    For $x_1$, we have $l_1$, $l_2$ and $m_2$ that happened earlier than $m_1$.
    So $W_1 = 1$ because $\operatorname{Ag}_{m_1}(M_1) = \{ M_1 \}$, $W_2 = 2$ because $\operatorname{Ag}_{m_1}(M_2) = \{M_2, M_3 \}$ by $m_2$, and $T = 1+ \binom{3}{1} = 4$ because the the horizontal PBB is $(L_1 (L_2 L_3))$ by $l_1$ and $l_2$, with every pair contributing to thickness.
    So $x_2 = 1 \cdot 2 \cdot 4 = 8$.
    
    \medskip
    
    \item For $x_2$, we have $l_2$ that happened earlier than $m_2$.
    So $W_1 = 1$ because $\operatorname{Ag}_{m_2}(M_2) = \{ M_2 \}$, $W_2 = 1$ because $\operatorname{Ag}_{m_2}(M_3) = \{ M_3 \}$, and $T = 1+1 = 2$ because the horizontal PBB is $L_1 (L_2 L_3)$ with $(L_2,L_3)$ being the pair that collided through $l_2$ and contributes to thickness.
    Thus $x_2 = 1 \cdot 1 \cdot 2 = 2$.
    
    \medskip
    
    \item For $y_1$, we have $l_2$ and $m_2$ that happened earlier than $l_1$.
    So $W_1 = 1$ because $\operatorname{Ag}_{l_1}(L_1) = \{ L_1 \}$, $W_2 = 2$ because $\operatorname{Ag}_{l_1}(L_2) = \{L_2, L_3 \}$ by $l_2$, and $T = 1+1 = 2$ because the vertical PBB is $M_1 (M_2 M_3)$ with $(M_2,M_3)$ being the pair that collided through $m_2$ and contributes to thickness.
    Thus $y_1 = 1 \cdot 2 \cdot 2 = 4$.
    
    \medskip
    
    \item
    For $y_2$, we have no collision that happened earlier than $l_2$.
    So $W_1 = 1$ because $\operatorname{Ag}_{l_2}(L_2) = \{ L_2 \}$, $W_2 = 1$ because $\operatorname{Ag}_{l_2}(L_3) = \{ L_3 \}$, and $T = 1$ because the vertical PBB is $M_1 M_2 M_3$ with no brackets.
    Thus $y_2 = 1 \cdot 1 \cdot 1 = 1$.
\end{enumerate}

\noindent
Thus $(\bx,\by)_v = (8,2,4,1)$.
\null\hfill$\triangle$
\end{eg}

Our main result is the following theorem:

\begin{theo}
The convex hull of the points $(\bx,\by)_v$, as $v$ varies over the vertices in $C(m,n)$, is a polytope whose face poset is isomorphic to $C(m,n)$.
\end{theo}

\begin{proof}
Combine Lemmas \ref{hyperplane}, \ref{facet_type_1}, \ref{facet_type_2}, and \ref{facet_type_3}.
\end{proof}

The lemmas used in the proof of this theorem implicitly describe the normal fan and the support function of our polytopal realization of $C(m,n)$.

\begin{lem}
\label{hyperplane}
For any $v \in C(m,n)$, the point $(\bx,\by)_v$ lies in the hyperplane 
\begin{align}
\left\{
(\bx,\by)
\in
\bR^{m-1}\times\bR^{n-1}
\:\left|\:
\sum_{i=1}^{m-1} x_i + \sum_{j=1}^{n-1} y_j
=
\binom{n}{2}\binom{m}{2} + \binom{n}{2} + \binom{m}{2}
\right.
\right\}.
\end{align}
\end{lem}

\begin{proof}
We first show that the equality holds for a certain vertex $v_0$ with most computable coordinates, and then show that the sum of  all coordinates doesn't change along an edge.

Let $v_0$ be a vertex with poset $l_1 < \cdots < l_{n-1} < m_1 < \cdots < m_{m-1}$ (first all horizontal lines are collapsed, left to right, then all vertical lines are collapsed, top to bottom).
Then the horizontal coordinates are $1, \ldots, m$ (all computed with thickness $1$), and the vertical coordinates are $1 \cdot (\binom{m}{2} + 1), \ldots, n \cdot \left(\binom{m}{2}+1\right)$ (all computed with thickness precisely $\binom{m}{2}+1$).
So the sum is indeed
\begin{align}
(1 + \cdots + m) + (1+ \cdots +n)\left(\binom{m}{2}+1\right)
=
\binom{n}{2}\binom{m}{2} + \binom{n}{2} + \binom{m}{2}.
\end{align}

Now recall Lemma \ref{connect}.
Let $E$ be some edge.
It corresponds to a preorder where some equivalence class $C$ has cardinality 2.
There are three possibilities: 
\begin{enumerate}
    \item
    $C = \{ l_a, l_b \} $ and $l_i<l_a \equiv l_b$ for $a < i < b$ (by {\sc (parallel comparability)}).
    
    \medskip
    
    \item
    $C = \{ m_a, m_b \} $ and $m_i<m_a \equiv m_b$ for $a < i < b$ (by {\sc (parallel comparability)}).
    
    \medskip
    
    \item
    $C = \{ l_a, m_b \} $, with no conditions.
\end{enumerate}

\noindent
For an edge of type 1, let $v$ be its endpoint with $l_a < l_b$ and let $w$ be its endpoint with $l_a > l_b$.
Then $v$ and $w$ only differ in two horizontal coordinates $y_a$ and $y_b$.
We notice that the thickness of the vertical PBB is the same for $l_a$ and $l_b$, no matter in which order they collide; we denote this quality by $T$.
For $v$, let denote $|\!\operatorname{Ag}_{l_a}(L_a)|$ by $A$, and for $w$, denote $|\!\operatorname{Ag}_{l_b}(L_{b+1})|$ by $B$.
In this notation, for $v$ we have
\begin{align}
y_a(v) = A(b-a)T,
\qquad
y_b(v) = (A + (b-a))BT,
\end{align}
and for $w$ we have
\begin{align}
y_a(w) =  A((b-a)+B)T,
\qquad
y_b(w) = (b-a)BT.
\end{align}
For both vertices, the sum of the two coordinates is equal to
\begin{align}
(A(b-a) + AB + (b-a)B)T.
\end{align}
For an edge of type 2, the argument is the same.

Finally, for an edge of type 3, let $v$ be its endpoint with $l_a < m_b$, and let $w$ be its endpoint with $l_a > m_b$.
Then $v$ and $w$ only differ in coordinates $y_a$ and $x_b$. Let $T_{\mathrm{vert}}$ be the thickness of the vertical PBB by the time of collision $l_a$ in $v$, and let $T_{\mathrm{hor}}$ be the thickness of the horizontal PBB by the time of collision $m_b$ in $w$.
We notice that $|\!\operatorname{Ag}_{l_a}(L_a)|$ and $|\!\operatorname{Ag}_{l_a}(L_{a+1})|$ are the same for $v$ and $w$; denote these qualities by $A_1$ and $A_2$.
Similarly, we notice that $|\!\operatorname{Ag}_{m_b}(M_b)|$ and $|\!\operatorname{Ag}_{m_b}(M_{b+1})|$ are the same for $v$ and $w$; denote these qualities by $B_1$ and $B_2$.
In this notation, for $v$ we have
\begin{align}
y_a(v)
=
A_1A_2T_{\mathrm{vert}},
\qquad
x_b(v)
=
B_1B_2(T_{\mathrm{hor}} + A_1A_2),
\end{align}
and for $w$ we have 
\begin{align}
y_a(w)
=
A_1A_2(T_{\mathrm{vert}}+B_1B_2),
\qquad
x_b(w)
=
B_1B_2T_{\mathrm{hor}}.
\end{align}
For both vertices, the sum of the two coordinates is equal to
\begin{align}
A_1A_2T_{\mathrm{vert}} + A_1A_2B_1B_2 +  B_1B_2T_{\mathrm{hor}}.
\end{align}
\end{proof}

To formulate the next Lemma, we need to classify facets of $C(m,n)$.
Recall that a facet is good rectangular preorder with two equivalence classes, $C_1$ and $C_2 = C(m,n) \setminus C_1$. For {\sc (orthogonal comparability)} and {\sc (parallel comparability)} to be satisfied, $C_1$ can be one of the following:

\begin{enumerate}
    \item
    $C_1 = \{l_i \:|\: i \in I \}$ where $I \subset [1,n-1]$ is a subinterval not equal to all of $[1,n-1]$.
    
    \medskip
    
    \item
    $C_1 = \{m_j \:|\: j \in J\}$ where $J \subset [1,m-1]$ is a subinterval not equal to all of $[1,m-1]$.
    
    \medskip
    
    \item
    $C_1 = \{l_i, m_j \:|\: i \in \bigcup_x I_s, j \in \bigcup_t J_t\}$ where $I_s$ are subintervals of $[1,n-1]$ satisfying $\max I_s < \min I_{s+1}$ and $J_t$ are subintervals of $[1,m-1]$ satisfying $\max J_t < \min J_{t+1}$.
\end{enumerate}

\noindent
We say that the corresponding facets are of types 1, 2, and 3 accordingly.

\begin{lem}
\label{facet_type_1}
Let $F \in C(m,n)$ be facet of type 1 with $C_1 = \{l_i \:|\: i \in I \}$.
Set $a \coloneqq |I|+1$.
Then for every vertex $v$ in  $F$ we have 
\begin{align}
\sum_{i \in I} y_i = \binom{a}{2}.
\end{align}
For vertex $w$ outside $F$ we have 
\begin{align}
\sum_{i \in I} y_i > \binom{a}{2}.
\end{align}
\end{lem}

\begin{proof}
Just as in the previous case, we first verify the equality for a vertex whose coordinates are most computable.
Set $r = \min I$, and let $v$ be the vertex with corresponding to the following order where $l$'s happen left to right first inside $C_1$, then inside $C_2$, and then $m$'s happen top to bottom: formally, $l_a < l_b$ either in one of the three cases: $a \in I$ and $b \notin I$, or $a, b \in I$ and $a < b$, or $a, b \notin I$ and $a<b$; $l_a <m_b$ always; and $m_a < m_b$ when $a < b$.
Then the collisions contributing to the coordinates $y_i$ for $i \in I$ all happen with thickness 1, and they are equal to $1$, $2$, $\ldots$, $a$, proving the equality for $v$.
To see that the equality holds for any vertex of $F$, we recall Lemma \ref{connect2}, consider an edge within $F$ and notice that the coordinate change described in the proof of Lemma \ref{hyperplane} happens only among $y_i$ with $i \in I$, thus not affecting $\sum_{i \in I} y_i$.

To prove the inequality, let $w$ be some vertex outside $F$.
Being outside $F$ means that there is a collision $c$ that happened earlier than all the collisions of $C_1$.
This collision contributes to some agglomeration size or some thickness among the coordinates coming from collisions of $C_1$.
Thus it makes $\sum_{i \in I} y_i$ strictly greater in $w$ than in a vertex of $F$ given by restricting the preorder of $w$ to $C_1$, and then ordering other collisions arbitrarily.
\end{proof}

\begin{lem}
\label{facet_type_2}
Let $F \in C(m,n)$ be facet of type 2 with $C_1 = \{m_i \:|\: j \in J \}$.
Denote $b = |J|+1$.
Then for every vertex $v < F$ we have 
\begin{align}
\sum_{j \in J} x_j = \binom{b}{2}.
\end{align}
Furthermore, for any vertex $v$ incomparable with $F$ we have
\begin{align}
\sum_{j \in J} x_j > \binom{b}{2}.
\end{align}
\end{lem}

\begin{proof}
The proof is identical to the proof of the previous lemma.
\end{proof}

\begin{lem}
\label{facet_type_3}
Let $F \in C(m,n)$ be facet of type 3 with $C_1 = \{l_i, m_j \:|\: i \in \bigcup_{s} I_s, j \in \bigcup_t J_t\}$.
Denote $a_s \coloneqq |I_s|+1$ and $b_t \coloneqq |J_t|+1$.
Then for every vertex $v < F$ we have 
\begin{align}
\sum_{i \in \bigcup_{s} I_s} y_i + \sum_{j \in \bigcup_{t} J_t} x_j =  \sum_s \binom{a_s}{2} + \sum_t \binom{b_t}{2} + \sum_{s,t} \binom{a_s}{2} \binom{b_t}{2}.
\end{align}

\noindent
Furthermore, for any vertex $v$ incomparable with $F$ we have
\begin{align}
\sum_{i \in \bigcup_{s} I_s} y_i + \sum_{j \in \bigcup_{t} J_t} x_j >  \sum_s \binom{a_s}{2} + \sum_t \binom{b_t}{2} + \sum_{s,t} \binom{a_s}{2} \binom{b_t}{2}.
\end{align}
\end{lem}

\begin{proof}
Just as in the previous case, we first verify the equality for a vertex whose coordinates are most computable.
Set $r_s = \min I_s$, $q_t = \min J_t$ and again let $v$ be the vertex with corresponding to the order where $l$'s happen left to right first inside $C_1$, then inside $C_2$, then other $l$'s happen left to right, then other $m$'s happen top to bottom.
Formally this means that for $l_i$ and $l_j$, $l_i < l_j$ holds in one of the following cases: $i,j \in L_s$ for some $s$, and $i < j$, or $i \in \bigcup_s I_s$ and $j \notin \cup_s I_s$, or $i,j \notin \bigcup_s I_s$ and $i < j$.
Parallel comparisons between $m$'s are identical.
For $l_i$ and $m_j$, $l_i > m_j$ holds if $i \notin \bigcup_s I_s$ and $j \in \bigcup_t J_t$.
Then the collisions contributing to the coordinates $y_i$ for $i \in I_s$ all happen with thickness 1, and they are equal to $1$, $2$, $\ldots$, $a_s$.
Collisions contributing to the coordinates $x_j$ for $k \in J_t$ happen with thickness $\sum_s \binom{a_s}{2}$, and their weight products are equal to $1$, $2$, $\ldots$, $b_t$.
This proves the equality for $v$.
To see that the equality holds for any vertex of $F$, we recall Lemma \ref{connect2}, consider an edge within $F$ and notice that the coordinate change described in the proof of Lemma \ref{hyperplane} happens only among the coordinates featured in the above sum, thus not affecting it.
\end{proof}

\noindent
This finishes the proof of the theorem.
A biproduct of this proof is the fact that constrainahedra are generalized permutahedra in the sense of \cite{postnikov:generalized_permutahedra}: all their edges have directions $e_i - e_j$, which is one of the characterizations of generalized permutahedra.
Another characterization of generalized permutahedra is that every chamber in their normal fan is a union of chambers in the braid arrangement, which correspond to linear orders.
This is also easily seen for constrainahedra: indeed, a chamber in the normal fan of a constrainahedron corresponds to a vertex good rectangular preorder, and it is a union of chambers in the braid arrangement corresponding to linear orders refining it.

\begin{rem}
For associahedra, the embedding presented in this section is the classical Loday embedding \cite{loday}.
For multiplihedra, the embedding presented in this section is the classical Forcey embedding \cite{forcey} for $q = 1/2$, scaled by $2$ and forced to live in hyperplane, while Forcey constructs a full-dimensional object.
\null\hfill$\triangle$
\end{rem}

\begin{rem}[higher constrainahedra]
All of the constructions in this paper can be easily generalized from dimension 2 to dimension $n$.
Geometrically, higher $n$-constrainahedra encode the collision in a grid consisting of orthogonal hyperplanes in $\mathbb{R}^n$.
So far, we were discussing $n = 2$, while $n = 1$ gives simply associahedra.

Let $\Coll(n_1, \ldots, n_k)$ be the set of collisions $m_i(j)$, where $j$ is allowed to vary from $1$ to $k$ and, for $j$ fixed, $i$ is allowed to vary from $1$ to $n_j-1$. Then the definition of a good $n$-rectangular preorder is repeated verbatim.
We leave the translation of \S\S\ref{sec:main_def}--\ref{sec:lattice} to an interested reader.
Finally, the modification of \S\ref{sec:polytopes} is that instead of just one thickness constant $T$ one is supposed to have $k-1$ of them, $T_1$ to $T_{k-1}$, and then total thickness is $T = T_1\cdots T_{k-1}$.
The proofs are then repeated verbatim.
\null\hfill$\triangle$
\end{rem}

\section{Relation to shuffle product realization}
\label{sec:comparison}

It is natural to wonder about the relationship between constrainahedra and products of two associahedra.
This relationship has been clarified by very recent work of Chapoton--Pilaud in \cite{chapoton_pilaud}: combinatorially, constrainahedra are \emph{shuffle products} of associahedra.

In \cite{chapoton_pilaud}, Chapoton--Pilaud defined the shuffle product of generalized permutahedra.
For such $P \subset \mathbb{R}^m$ and $Q \subset \mathbb{R}^n$, their shuffle product is their direct product followed by a Minkowski sum with a certain zonotope:
\begin{align}
P \star Q
\coloneqq
P \times Q + \sum_{i \in [m], j \in [n]} [e_i,e_{m+j}].
\end{align}
Chapoton--Pilaud showed that as posets, 
\begin{align}
C(m,n)
\simeq
A(m) \star A(n).
\end{align}

However, geometrically the realization in the current paper differs from that of \cite{chapoton_pilaud}.
We now explain the similarity and difference for vertex coordinates. In both papers, a vertex coordinate is expressed via three numbers: two weights $W_1$ and $W_2$, and some third number representing, informally, what happened earlier. In our paper, the formula is
\begin{align}
W_1 \times W_2 \times T.
\end{align}
In \cite{chapoton_pilaud}, the formula is
\begin{align}
W_1 \times W_2 + \widetilde{T},
\end{align}
where $\widetilde{T}$ is the number of nodes in the same PBB that we use to define $T$.

For comparison, consider again the vertex whose coordinates we have computed to be $(x_1,x_2,y_1,y_2) = (8,2,4,1)$:

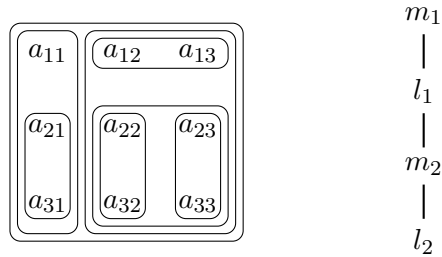
\begin{center}
\begin{tikzpicture}

\node (a11) {$a_{11}$};
\node (a12) [right of=a11] {$a_{12}$};
\node (a13) [right of=a12] {$a_{13}$};

\node (a21) [below of =a11] {$a_{21}$};
\node (a22) [right of=a21] {$a_{22}$};
\node (a23) [right of=a22] {$a_{23}$};

\node (a31) [below of =a21] {$a_{31}$};
\node (a32) [right of=a31] {$a_{32}$};
\node (a33) [right of=a32] {$a_{33}$};

\draw[rounded corners] ($(a21)+(-0.3,0.2)$) rectangle ($(a31)+(0.3,-0.2)$);
\draw[rounded corners] ($(a22)+(-0.3,0.2)$) rectangle ($(a32)+(0.3,-0.2)$);
\draw[rounded corners] ($(a23)+(-0.3,0.2)$) rectangle ($(a33)+(0.3,-0.2)$);

\draw[rounded corners] ($(a22)+(-0.4,0.3)$) rectangle ($(a33)+(0.4,-0.3)$);
\draw[rounded corners] ($(a12)+(-0.4,0.2)$) rectangle ($(a13)+(0.4,-0.2)$);

\draw[rounded corners] ($(a11)+(-0.4,0.3)$) rectangle ($(a31)+(0.4,-0.4)$);
\draw[rounded corners] ($(a12)+(-0.5,0.3)$) rectangle ($(a33)+(0.5,-0.4)$);

\draw[rounded corners] ($(a11)+(-0.5,0.4)$) rectangle ($(a33)+(0.6,-0.5)$);

\begin{scope}[shift = {(5,0.5)}]
\node (1) {$m_1$};
\node (2) [below of = 1] {$l_1$};
\node (3) [below of = 2] {$m_2$};
\node (4) [below of = 3] {$l_2$};

\draw[thick] (1) -- (2) -- (3) -- (4);
\end{scope}

\end{tikzpicture}
\end{center}

In the realization defined in \cite{chapoton_pilaud}, the same vertex has coordinates $(x_1,x_2,y_1,y_2) = (4,2,3,1)$:
\begin{align}
x_1 = 4 = 1 \times 2 + 2,
\qquad
x_2 = 2 = 1 \times 1 + 1,
\qquad
y_1 = 3 = 1 \times 2 + 1,
\qquad
y_2 = 1 = 1 \times 1 + 0.
\end{align}

\bibliographystyle{alpha}
\small
\bibliography{biblio}

\end{document}